\documentclass[11pt]{amsart}

\usepackage{amsfonts,latexsym}
\usepackage{amsmath}
\usepackage{amstext}
\usepackage{amssymb}
\usepackage{color}
\usepackage{graphicx}
\graphicspath{{figures/}}

\newcommand{\R}{\mathbb{R}}

\newtheorem{theorem}{Theorem}[section]
\newtheorem{lemma}[theorem]{Lemma}
\newtheorem{proposition}[theorem]{Proposition}

\theoremstyle{definition}
\newtheorem{definition}[theorem]{Definition}

\theoremstyle{remark}
\newtheorem{remark}[theorem]{Remark}

\numberwithin{equation}{section}

\begin{document}

\title[Polytopic Invariants for Polynomial Dynamical Systems]
      {Computation of Polytopic Invariants for Polynomial Dynamical Systems using Linear Programming}
\thanks{This work was supported by the Agence Nationale de la Recherche (VEDECY project - ANR 2009 SEGI 015 01).}


\author[Mohamed Amin Ben Sassi]{Mohamed Amin Ben Sassi}
\address{Laboratoire Jean Kuntzmann \\
Universit\'e de Grenoble \\
B.P. 53, 38041 Grenoble, France}
\email{Mohamed-Amin.Bensassi@imag.fr}

\author[Antoine Girard]{Antoine Girard}
\address{Laboratoire Jean Kuntzmann \\
Universit\'e de Grenoble \\
B.P. 53, 38041 Grenoble, France} \email{Antoine.Girard@imag.fr}


\maketitle


\begin{abstract}
This paper deals with the computation of polytopic invariant sets for polynomial dynamical systems.
An invariant set of a dynamical system is a subset of the state space such that if the state of the system belongs to the set at a given instant, 
it will remain in the set forever in the future. Polytopic invariants for polynomial systems can be verified by solving a set of optimization problems involving multivariate polynomials on bounded polytopes. 
Using the blossoming principle together with properties of multi-affine functions on rectangles and Lagrangian duality, we show that certified lower bounds of the optimal values of such optimization problems can be computed effectively using linear programs. 
This allows us to propose a method based on linear programming for verifying polytopic invariant sets of polynomial dynamical systems. Additionally,  
using sensitivity analysis of linear programs, one can iteratively compute a polytopic invariant set. 
Finally, we show using a set of examples borrowed from biological applications, that our approach is effective in practice.
\end{abstract}

\section{Introduction}

An invariant set of a dynamical system is a subset of the state space such that if the state of the system belongs to the set at a given instant, 
it will remain in the set forever in the future. Invariant sets are fundamental notions in dynamical systems theory where they can serve to prove the
existence of attractors (e.g. in the Poincar\'e Bendixon theorem~\cite{Sastry}); they have played an important role
in control theory for the analysis of performance, robustness or practical stability~\cite{blanchini99}. 
They are also of great interest for 
reachability analysis of continuous and hybrid systems, especially for verification of safety properties where the goal is to prove that trajectories of a  system starting from a given set of initial states will never reach a specified set of unsafe states~\cite{Alur2011}.
This can be done by exhibiting an invariant set, containing the set of initial states, and whose intersection with the set of unsafe states is empty. For polynomial dynamics, these invariants are often given by semi-algebraic sets~\cite{Prajna07,Platzer,Sankaranarayanan10}.
However, when the computation of an invariant set is part of a bigger process such as controller synthesis or safety verification, it is sometimes preferable
to have invariants given by polytopes that are easier to manipulate~\cite{Alessio2007,Sankaranarayanan2008}. 
For instance, for specific classes of polynomial systems such as multi-affine or quasi multi-affine systems, methods to obtain rectangular 
invariants have been developed in~\cite{Belta06,Abate09}. 

In this paper, we deal with the computation of polytopic invariant sets of polynomial dynamical systems. Let us remark that rectangles form a subclass  of polytopic invariants and in some sense, our work extends the work of~\cite{Belta06,Abate09}. 
More precisely, we shall consider a dynamical system of the form:
\begin{equation}
\label{eq:ode}
\dot x(t) = f(x(t)),\; x(t) \in \R^n
\end{equation}
where $f:\R^n \rightarrow \R^n$ is a polynomial vector field. We consider the dynamics of (\ref{eq:ode}) only on a bounded rectangle $R$ of the state
space $\R^n$;
given a bounded polytope $P\subseteq R$ with a set of facets $\{F_k|\; k\in K\}$, it follows from the standard characterization of invariant sets (see e.g.~\cite{aubin1991}) that $P$ is invariant for the dynamical system~(\ref{eq:ode}) if and only if
\begin{equation}
\label{eq:car}
\forall k\in K,\; \forall x\in F_k,\; a_k\cdot f(x) \le 0
\end{equation}
where $a_k$ is the normal vector to $F_k$ pointing outside $P$. As pointed out in~\cite{Abate09} and by application of Tarski's Theorem~\cite{tarski1948}, this a decidable problem. However, the complexity of the decision procedure gives little hope for practical application. Let us remark that (\ref{eq:car}) can be reformulated as follows:
\begin{equation}
\label{eq:car2}
\forall k\in K,\; \min_{x\in F_k} -a_k\cdot f(x) \ge 0.
\end{equation}
This consists in showing that the minimal values of the multivariate polynomials  $-a_k\cdot f$ on the bounded polytopes $F_k$ are positive. Hence, if we are able to compute non-negative certified lower bounds of these minimal values, it is sufficient to prove that the polytope $P$ is invariant for the dynamical system~(\ref{eq:ode}). 

In this paper, we establish linear programming (LP) relaxations of the optimization problems in (\ref{eq:car2}). 
The main tool we use is the blossoming principle (see e.g.~\cite{Seidel1993}) that essentially maps the set of polynomials to the set of symmetric multi-affine functions. 
The blossoming principle together with properties of multi-affine maps on bounded rectangles allows us to derive linear programs which makes it possible to compute
lower bounds of the minimal values in (\ref{eq:car2}). Our approach is conservative (the lower bound is not tight) but it is effective and may be sufficient for proving invariance of a polytope. 
Additionally, we will show how one can iteratively compute a polytopic invariant set using sensitivity analysis of linear programs. 
Finally, we show using a set of examples borrowed from biological applications, that our approach is effective in practice.

\section{Preliminaries}

In this section, we introduce notations and preliminary results that will be useful for subsequent discussions.

\subsection{Multi-affine functions}

Multi-affine functions form a particular class of multivariate polynomials.
Essentially, a multi-affine function is a function which is affine in each of its variables when the other variables are regarded as constant:
\begin{definition} A multi-affine function $p:\R^n \rightarrow \R$ is a multivariate polynomial in the variables $x_1,\dots,x_n$ where the degree of $p$ in each of the variable is at most $1$. For $x=(x_1,\dots,x_n)$,
$$
p(x)=\sum_{(l_1,\dots,l_n)\in \{0,1\}^n } p_{l_1,\dots,l_n} x_1^{l_1}\dots x_n^{l_n}
$$
where $p_{l_1,\dots,l_n} \in \R$ for all $(l_1,\dots,l_n)\in \{0,1\}^n$.
\end{definition}

Let $R=\prod_{k=1}^{k=n}[\underline{x}_k,\overline{x}_k]$ be a rectangle of $\R^n$, with $\underline{x}_k<\overline{x}_k$, for all $k\in \{1,\dots, n\}$; the set of vertices of $R$ is 
$V=\prod_{k=1}^{k=n}\{\underline{x}_k,\overline{x}_k\}$.
It is shown in~\cite{Belta06} that a multi-affine function $p$ is uniquely determined by its values at the vertices of a rectangle $R$. Moreover, for all $x\in R$, $p(x)$ is a convex combination of the values at the vertices, that is $p(R)\subseteq CH(\{p(v) |\; v\in V\})$ where $CH(S)$ denotes the convex hull of the set $S$. Let us remark that generally $p(R)$ is not convex and  $P(R) \ne CH(\{p(v) |\; v\in V\})$. Then, we have the following result:
\begin{lemma}\label{lem:conv}
Let $p$ be a multi-affine function and $R$ a rectangle with set of vertices $V$, then
$
\displaystyle{\min_{x\in R} p(x)= \min_{v\in V} p(v).}
$
\end{lemma}

The previous result shows that optimizating a multi-affine function $p$ over a rectangle $R$ only requires finding the minimal value of $p$ at the vertices of $R$.

\subsection{Blossoming principle}
\label{sec:blossom}
Let
$p:\R^n \rightarrow \R$ be an arbitrary multivariate polynomial function,
let $\delta_1,\dots,\delta_n$ denote the degree of $p$ in the variables $x_1,\dots,x_n$ respectively. 
Let $\Delta=\{0,\dots,\delta_1\} \times \dots \times \{0,\dots,\delta_n\}$, then $p(x)$ can be written under the form:
$$
p(x)=\sum_{(l_1,\dots,l_n)\in \Delta } p_{l_1,\dots,l_n} x_1^{l_1}\dots x_n^{l_n}
$$
where $p_{l_1,\dots,l_n} \in \R$ for all $(l_1,\dots,l_n)\in \Delta$.
We now present the blossoming principle which consists in mapping polynomials to symmetric multi-affine functions. Blossoms have been developed in the area of computer aided geometric design where they have numerous applications, most notably for spline curves and surfaces. 
As for our problem, the blossoming principle allows us to recast the optimization of polynomial functions as the optimization of a multi-affine functions for which we can use the fundamental property presented in Lemma~\ref{lem:conv}. 
All the results in this section are quite standard (see~\cite{Seidel1993} and references therein) and are therefore stated without proofs.
\begin{definition}\label{def:blossom} The blossom or polar form of the polynomial $p:\R^n \rightarrow \R$ is the function 
$q:\R^{\delta_1+\dots+\delta_n} \rightarrow \R$
given for $z=(z_{1,1},\dots,z_{1,\delta_1},\dots,z_{n,1},\dots,z_{n,\delta_n})$ by
$$
q(z)=\sum_{(l_1,\dots,l_n)\in \Delta } p_{l_1,\dots,l_n} \prod_{i=1}^{i=n} B_{l_i,\delta_i}(z_{i,1},\dots,z_{i,\delta_i})
$$
with
$$
B_{l,\delta}(z_1,\dots,z_\delta) = \frac{1}{\left(\begin{smallmatrix} \delta \\ l \end{smallmatrix} \right)} \sum_{\sigma \in C(l,\delta)} z_{\sigma_1} \dots z_{\sigma_l}
$$
where $C(l,\delta)$ denotes the set of combinations of $l$ elements in  $\{1,\dots,\delta\}$.
\end{definition}

An example may help to understand the definition;
the blossom of the polynomial $p(x)=3x_1+2x_2^3+x_1^2x_2^2$ is 
$$
\begin{array}{lcl}
q(z) &=& \frac{3}{2} (z_{1,1}+z_{1,2})+ 2 z_{2,1}z_{2,2}z_{2,3}\\
&&  + \frac{1}{3} z_{1,1}z_{1,2}(z_{2,1}z_{2,2}+z_{2,1}z_{2,3}+z_{2,2}z_{2,3}).
\end{array}
$$

We define a relation on $\R^{\delta_1+\dots+\delta_n}$:
for $z,z' \in \R^{\delta_1+\dots+\delta_n}$, with 
$z=(z_{1,1},\dots,z_{1,\delta_1},\dots,z_{n,1},\dots,z_{n,\delta_n})$  and 
$z'=(z'_{1,1},\dots,z'_{1,\delta_1},$ $\dots,z'_{n,1},\dots,z'_{n,\delta_n}),$ 
we denote
$z \cong z'$ if, for all $k=1,\dots,n$, there exists a permutation $\pi_k$
such that $(z_{k,1},\dots,z_{k,\delta_k})= \pi_k(z'_{k,1},\dots,z'_{k,\delta_k})$.
It is easy to see that $\cong$ is an equivalence relation. 
A characterization of blossoms that is equivalent to Definition~\ref{def:blossom} is given by the following proposition:
\begin{proposition}\label{pro:blossom}
$q:\R^{\delta_1+\dots+\delta_n} \rightarrow \R$ is the blossom of the polynomial $p:\R^n \rightarrow \R$ if and only if:
\begin{enumerate}
\item $q$ is a multi-affine function;

\item $q$ is a symmetric function of its arguments: 
$$ 
\forall z\cong z',\; q(z)=q(z');
$$

\item $q$ satisfies the diagonal property:
$$
q(z_1,\dots,z_1,\dots,z_n,\dots,z_n) = p(z_1,\dots,z_n).
$$
\end{enumerate}
\end{proposition}

Let $R=\prod_{k=1}^{k=n}[\underline{x}_k,\overline{x}_k]$ be a rectangle of $\R^n$, with $\underline{x}_k<\overline{x}_k$, for all $k\in \{1,\dots, n\}$; we define the associated rectangle of $\R^{\delta_1+\dots+\delta_n}$ defined as $R'=\prod_{k=1}^{k=n}[\underline{x}_k,\overline{x}_k]^{\delta_k}$ and its set of vertices $V'=\prod_{k=1}^{k=n}\{\underline{x}_k,\overline{x}_k\}^{\delta_k}$. 
For $v = (v_{1,1},\dots,v_{1,\delta_1},\dots,v_{n,1},\dots,v_{n,\delta_n})\in V'$ and
$k\in \{1,\dots,n\}$, $l_k(v)$ denotes the number of elements $v_{k,1},\dots,v_{k,\delta_k}$ that are equal to 
$\overline{x}_k$. It is easy to verify that for $v,v'\in V'$, $v\cong v'$ if and only if $l_k(v)=l_k(v')$ for all 
$k\in  \{1,\dots,n\}$. We denote by $\overline{V}'=(V'/\cong)$
the set of equivalence classes of the relation $\cong$ on the set $V'$;
$\overline{V}'$ has $(\delta_1+1)\times \dots \times (\delta_n+1)$ elements. 

From the previous discussion, for $k\in  \{1,\dots,n\}$ and $\overline v\in \overline V'$,  
$l_k(v)$ has the same value for all $v \in \overline v$, with a slight abuse of notation we denote this value $l_k(\overline v)$.
Also from the second property in Proposition~\ref{pro:blossom},for  $\overline v\in \overline V'$,  
$q(v)$ has the same value for all $v \in \overline v$, let us denote this value $q(\overline v)$.


\begin{proposition}
\label{pro:bernstein}
The values $q(\overline v)$ for $\overline v\in \overline V'$
are the coordinates of the polynomial $p$ in the Bernstein basis:
$$
p(x) = \sum_{\overline v\in \overline V'} q(\overline v) \prod_{k=1}^{k=n} \mathcal B_{l_k(\overline v),\delta_k}
\left(\frac{x-\underline x_k}{\overline x_k - \underline x_k} \right)
$$
where $\mathcal B_{l,\delta}$ are the Bernstein polynomials
$$
\mathcal B_{l,\delta}(y) = \left(
\begin{smallmatrix}
\delta \\
l
\end{smallmatrix}
 \right) y^l (1-y)^{\delta-l}, \; l\in \{1,\dots,\delta\}.
$$
\end{proposition}

The previous result can be useful when one needs to compute the values $q(\overline v)$ for $\overline v \in \overline V'$.
The explicit computation of the blossom $q(z)$ (which can count up to 
$2^{\delta_1+\dots+\delta_n}$ terms) is computationally expensive. However,
Proposition~\ref{pro:bernstein} states that
 it is sufficient
to compute the $(\delta_1+1)\times \dots \times (\delta_n+1)$
coordinates of the polynomial $p(x)$ in the Bernstein basis. 
This can be done simply by solving a system of linear equations.

\section{LP Relaxations for Optimization of Polynomials on Polytopes}
\label{sec:opt}
As stated in the introduction, the verification of polytopic invariants for polynomial dynamical systems can be handled by solving a set of problems of optimization of multivariate polynomials on bounded polytopes given by (\ref{eq:car2}). Therefore, in this section, we consider the following problem:
\begin{equation}
\label{eq:opt}
\begin{array}{llr}
\text{min} & p(x)\\
\text{over} & x\in R, \\
\text{under} 
& a_i \cdot x \le b_i, & i\in I, \\
& c_j \cdot x = d_j, & j\in J.
\end{array}
\end{equation}
where $p:\R^n \rightarrow \R$ is a multivariate polynomial, 
$R$ is a rectangle of $\R^n$ with set of vertices $V$; 
$I=\{1,\dots,m_I\}$ and $J=\{1,\dots,m_J\}$ are sets of indices; $a_i\in \R^n$, $b_i\in \R$, for all $i\in I$ and $c_j\in \R^n$,  $d_j\in \R$, for all $j\in J$.
Let us remark that even though the polytope defined by the constraints indexed by $I$ and $J$ is unbounded in $\R^n$, the fact that we consider $x\in R$ which is a bounded rectangle of $\R^n$ results in an optimization problem on a bounded (not necessarily full dimensional) polytope of $\R^n$. 
We will assume that the problem is feasible: there exists $x\in R$ satisfying all the constraints.
Problems defined in (\ref{eq:car2}) can be recasted under the form (\ref{eq:opt}) with polynomial $p(x)=-a_k\cdot f(x)$ and linear inequality and equality constraints indexed  by $I$ and $J$  describing the facet $F_k$.

As the function $p$ is usually non-convex, this may be a non-trivial problem to solve.
Let us remark that as far as verification of polytopic invariants is concerned, we are not interested in computing the solution of problem~(\ref{eq:opt}) (i.e. $x^* \in R$ satisfying constraints and minimizing $p$).
Indeed, it is sufficient to compute the optimal value of (\ref{eq:opt}), that is $p^*=p(x^*)$, or at least a certified lower bound of the optimal value.
We will first show how this can be done if $p$ belongs to the particular class of multi-affine functions. Then, we will extend these results to arbitrary polynomial functions.

\subsection{Optimization of multi-affine functions}
\label{sec:multi-affine}

In the following, we show how to compute,  using linear programming, a certified lower bounded of the optimal value $p^*$ of (\ref{eq:opt}) where $p$ is a multi-affine function. The linear program is derived through Lagrangian duality.
We start by writing the Lagrangian of problem (\ref{eq:opt}):
$$
L(x,\lambda,\mu) = p(x)+\sum_{i\in I} \lambda_i (a_i \cdot x - b_i) + \sum_{j\in J} \mu_j (c_j \cdot x - d_j)
$$
where $x\in R$, $\lambda_i \ge 0$ for all $i\in I$, and $\mu_j\in \R$ for all $j\in J$.
Then, the dual formulation of problem~(\ref{eq:opt}) is 
\begin{equation}
\label{eq:dual}
\begin{array}{llr}
\text{max} & \displaystyle{\min_{x\in R} L(x,\lambda,\mu)}\\
\text{over} & \lambda\in \R^{m_I},\; \mu \in \R^{m_J},\\
\text{under} & \lambda_i \ge 0, & i\in I.
\end{array}
\end{equation}
Since (\ref{eq:opt}) is feasible, the optimal value of (\ref{eq:dual}) is bounded, it is denoted $d^* \in \R$.
It is well-known from duality theory  (see e.g.~\cite{boyd2004}) that we have $d^* \le p^*$. 
A multi-affine function is generally non-convex; then, we cannot expect strong duality (i.e.~$d^*=p^*$) in general. The following result shows that
 problem (\ref{eq:dual}) can be recasted as a linear program:
\begin{proposition}\label{pro:dual} Let $p:\R^n \rightarrow \R$ be a multi-affine function, the dual of problem (\ref{eq:opt}) is equivalent to the linear program:
\begin{equation}
\label{eq:dual1}
\begin{array}{lll}
\text{max} & t\\
\text{over} & t\in \R,\; \lambda\in \R^{m_I},\; \mu \in \R^{m_J},\\
\text{under} &\lambda_i \ge 0, & i\in I, \\ 
&t \le p(v)+\displaystyle{\sum_{i\in I} \lambda_i (a_i \cdot v - b_i)} \\ & 
\hspace{0.5cm}+ \displaystyle{\sum_{j\in J} \mu_j (c_j \cdot v - d_j)}, & v\in V.
\end{array}
\end{equation}
\end{proposition}

\begin{proof} Let us remark that the Lagrangian $L(x,\lambda,\mu)$ is a multi-affine function of $x$. Then, it follows from Lemma~\ref{lem:conv} that 
the minimum of $L(x,\lambda,\mu)$ over $x\in R$ is reached at one of vertices of $R$:
$$ 
\min_{x\in R} L(x,\lambda,\mu) = \min_{v\in V}  L(v,\lambda,\mu).
$$
Then, problem (\ref{eq:dual}) consist in optimizing a piecewise linear function under a set of linear constraints which it is straightforward (see~\cite{boyd2004}, pp 150-151) to formulate as the linear program~(\ref{eq:dual1}).
\end{proof}


We have presented a simple approach to compute a certified lower bound of the solution $p^*$ of~(\ref{eq:opt}). Though conservative, our approach relies on linear programming and is therefore effective.

\subsection{Optimization of polynomial functions}
\label{sec:optpol}
We consider problem (\ref{eq:opt}) where $p:\R^n \rightarrow \R$ is now
an arbitrary multivariate polynomial function.
We use the blossoming principle to define a problem equivalent to (\ref{eq:opt}) and involving $q$ the blossom of $p$.
We use the same notations as in section~\ref{sec:blossom}. Then, from the third property in Proposition~\ref{pro:blossom}, problem~(\ref{eq:opt}) is  equivalent to
\begin{equation}
\label{eq:opt1}
\begin{array}{lll}
\text{min} & q(z)\\
\text{over} & z\in R', \\
\text{under} 
& {a_i}' \cdot z \le b_i, & i\in I, \\
& {c_j}' \cdot z = d_j, & j\in J,\\
& e_{k,l} \cdot z = 0,& k\in\{1,\ldots,n \},\\
&& l \in \{1,\ldots,\delta_{k}-1 \}.
\end{array}
\end{equation}
where $a_i'={(\frac{a_{i,1}}{\delta_1},\dots,\frac{a_{i,1}}{\delta_1},\dots,\frac{a_{i,n}}{\delta_n},\dots,\frac{a_{i,n}}{\delta_n})}$, for $i\in I$;
$c_j' ={(\frac{c_{j,1}}{\delta_1},\dots,\frac{c_{j,1}}{\delta_1},\dots,\frac{a_{j,n}}{\delta_n},\dots,\frac{a_{j,n}}{\delta_n})}$, for $j\in J$; 
and $e_{k,l} \in \R^{\delta_1+\dots+\delta_n}$ are the vectors such that
$e_{k,l}\cdot z = z_{k,l}- z_{k,l+1}$, for all $z\in \R^{\delta_1+\dots+\delta_n}$, 
$k\in\{1,\ldots,n \}$, $l~\in~\{1,\ldots,\delta_{k}-1 \}$.

Now, since the blossom of a multivariate polynomial is a multi-affine function, we can remark that problem
$(\ref{eq:opt1})$ is similar to those considered in Section~\ref{sec:multi-affine}.
Then, we can use Proposition~\ref{pro:dual} to obtain its dual, which is given by the following linear program:
\begin{equation}
\label{eq:dual11}
\begin{array}{lll}
\text{max} & t\\
\text{over} & t\in \R,\; \lambda\in \R^{m_I},\; \mu \in \R^{m_J},\\
& \alpha \in \R^{(\delta_1-1)+\dots+(\delta_n-1)},\\
\text{under} &\lambda_i \ge 0, & i\in I, \\ 
&t \le q(v)+\displaystyle{\sum_{i\in I} \lambda_i (a_i' \cdot v - b_i)}\\
&  \hspace{0.5cm} + \displaystyle{\sum_{j\in J} \mu_j (c_j' \cdot v - d_j)} \\
&  \hspace{0.5cm} + \displaystyle{\sum_{k\in\{1,\ldots,n\}}{\sum_{l\in\{1,\ldots,\delta_k-1\}} \alpha_{k,l} (e_{k,l} \cdot v )}}, & v\in V'.\\
\end{array}
\end{equation}
By Proposition~\ref{pro:dual}, the optimal value of this linear program gives a certified lower bound of the optimal value
of the polynomial optimization problem (\ref{eq:opt}). However, we shall not solve directly the linear program (\ref{eq:dual11}) 
as the equivalence relation $\cong$ defined in section~\ref{sec:blossom} can be used to exploit the symmetries and reduce the complexity of the linear program. Let us remark that for all $v\cong v'$, $i\in I$ and $j\in J$,
$a_i'\cdot v = a_i'\cdot v'$ and $c_j'\cdot v = c_j'\cdot v'$.
Then, for  $\overline v\in \overline V'$,  
$a_i'\cdot v$ (respectively $c_j'\cdot v$) has the same value for all $v \in \overline v$, let us denote this value 
$a_i'\cdot \overline{v}$ (respectively $c_j'\cdot \overline{v}$).

\begin{theorem}
\label{th:dual21}
Let $p:\R^n \rightarrow \R$ be a polynomial and $q:\R^{\delta_1+\dots+\delta_n} \rightarrow \R$ 
its blossom. 
The optimal value of the linear program (\ref{eq:dual11}) is equal to the optimal value $d^*$ of:\\
\begin{equation}
\label{eq:dual21}
\begin{array}{lll}
\text{max} & t\\
\text{over} & t\in \R,\; \lambda\in \R^{m_I},\; \mu \in \R^{m_J},\\
\text{under} &\lambda_i \ge 0, & i\in I, \\ 
&t \le q(\overline{v})+\displaystyle{\sum_{i\in I} \lambda_i (a_i' \cdot \overline{v} - b_i)} \\
& \hspace{0.5cm}       + \displaystyle{\sum_{j\in J} \mu_j (c_j' \cdot \overline{v} - d_j)}, & \overline{v}\in \overline{V}'.
\end{array}
\end{equation}
Moreover, $d^* \le p^*$ where $p^*$ is the optimal value of (\ref{eq:opt}).
\end{theorem}
The proof is stated in appendix. Let us highlight the gain of solving (\ref{eq:dual21}) in place of 
(\ref{eq:dual11}). Firstly, the decision variables $\alpha_{k,l}$ in problem (\ref{eq:dual11}) do not appear anymore
in (\ref{eq:dual21}). Secondly, the number of constraints indexed by $v'\in V'$ in (\ref{eq:dual11}) is $2^{\delta_1+\dots+\delta_n}$ whereas the number of constraints indexed by $\overline{v}\in \overline{V}'$
is only $(\delta_1+1)\times \dots \times (\delta_n+1)$. In fact, the linear program 
(\ref{eq:dual11}) has $m_I+m_J+(\delta_1-1)+\dots+(\delta_n-1)+1$ variables and $2^{\delta_1+\dots+\delta_n}+m_I$ inequality constraints whereas the linear program~(\ref{eq:dual21}) has only $m_I+m_J+1$ variables and $(\delta_1+1)\times \dots \times (\delta_n+1)+m_I$ inequality constraints.  
As for the computation of the values $q(\overline v)$, in order to keep the computational cost as lows as possible, 
one should avoid computing explicitly the blossom of $p$; it is better to use the method suggested by Proposition~\ref{pro:bernstein}.
This way, the overall cost of the optimization procedure will remain polynomial in the degrees of $p$, though exponential in the dimension of the state space $\R^n$.

\subsection{Sensitivity analysis} An interesting feature of Lagrangian duality is that it enables sensitivity analysis (see e.g.~\cite{boyd2004}). In this section, we are interested in analyzing the variations of the optimal value of (\ref{eq:opt}), or of its lower bound, 
under modifications of the polytope. This will be used in the next section for the computation of polytopic invariants for polynomial dynamical systems. More precisely, we consider the following variation of problem (\ref{eq:opt}):
\begin{equation}
\label{eq:optper}
\begin{array}{llr}
\text{min} & p(x)\\
\text{over} & x\in R, \\
\text{under} 
& a_i \cdot x \le b_i+\alpha_i, & i\in I, \\
& c_j \cdot x = d_j+\beta_j, & j\in J,
\end{array}
\end{equation}
where $\alpha_i \in \mathbb{R}$, for all $i \in I$ and $=\beta_j \in \mathbb{R}$ for all $j \in J$.
This problem coincides with the original problem (\ref{eq:opt}) when $\alpha=0$ and $\beta=0$.
We assume that problem (\ref{eq:optper}) is feasible as well. Let $p^*$ and $p^*(\alpha,\beta)$ denote the optimal values of problems (\ref{eq:opt}) and (\ref{eq:optper}), respectively.
Let $d^*$ and $d^*(\alpha,\beta)$ be the lower bounds of $p^*$ and $p^*(\alpha,\beta)$ obtained by application of Theorem~\ref{th:dual21}. The following result shows how the solution of (\ref{eq:dual21})
allows us to compute a lower bound of $d^*(\alpha,\beta)$ and thus of $p^*(\alpha,\beta)$.

\begin{theorem}
\label{th:sens}
Let $d^*$ and $(t^*,{\lambda}^*,{\mu}^*)$ be the optimal value and an optimal solution of the linear program (\ref{eq:dual21}).
Then, for all $\alpha\in \R^{m_I}$ and $\beta\in \R^{m_J}$, such that (\ref{eq:optper}) is feasible we have:
$
p^*(\alpha,\beta) \geq d^*(\alpha,\beta) \geq d^* - {\lambda^*}\cdot \alpha- {\mu^*}\cdot \beta. 
$
\end{theorem}

\begin{proof} By applying Theorem~\ref{th:dual21} to (\ref{eq:optper}), we have that $d^*(\alpha,\beta)$ is the optimal value of
\begin{equation}
\label{eq:dual21per}
\begin{array}{lll}
\text{max} & t\\
\text{over} & t\in \R,\; \lambda\in \R^{m_I},\; \mu \in \R^{m_J},\\
\text{under} &\lambda_i \ge 0, & i\in I, \\ 
&t \le q(\overline{v})+\displaystyle{\sum_{i\in I} \lambda_i (a_i' \cdot \overline{v} - b_i-\alpha_i)}\\
& \hspace{0.5cm} + \displaystyle{\sum_{j\in J} \mu_j (c_j' \cdot \overline{v} - d_j-\beta_j)}, & \overline{v}\in \overline{V}'.
\end{array}
\end{equation}
The fact that $p^*(\alpha,\beta) \geq d^*(\alpha,\beta)$ is a consequence of Theorem~\ref{th:dual21}.
Let $(t^*,{\lambda}^*,{\mu}^*)$ be an optimal solution of the problem (\ref{eq:dual21}), 
one can verify easily that $(t^*- {\lambda^*}\cdot \alpha- {\mu^*}\cdot \beta, \lambda^*,\mu^*)$
is feasible for (\ref{eq:dual21per}). It follows that 
$d^*(\alpha,\beta) \geq t^* - {\lambda^*}\cdot \alpha- {\mu^*}\cdot \beta$ which leads to the expected inequality since $d^*=t^*$.
\end{proof}

\subsection{Examples and comparison}

Before using the method desribed previously for the verification and the computation of polytopic invariants for polynomial dynamical systems,
we provide a brief comparison with two existing relaxation methods for polynomial optimization:
the first one is the Reformulation Linearisation Technique (RLT) introduced by Sherali in~\cite{sherali91,sherali97} 
which is also based on linear programming; the second one was introduced by Lasserre in~\cite{lasserre2001} and uses ralexations uner the form of Linear Matrix Inequalities (LMI) which are solved using semi definite programming.
We compare the three methods using two different examples.

We first consider the following constrained 3-dimensional problem~\cite{sherali91}:
\begin{equation}
\label{eq:test2}
\begin{array}{lll}
\text{min} & x_1x_2x_3+{x_1}^2-2x_1x_2-3x_1x_3\\
           &+5x_2x_3-{x_3}^2+5x_2+x_3 \\
\text{over} & x=(x_1,x_2,x_3)\in [2,5]\times[0,10]\times[4,8], \\
\text{under} 
& 4x_1+3x_2+x_3 \le 20, \\
& x_1+2x_2+x_3 \ge 1.
\end{array}
\end{equation}
Optimal value and solution of this problem are $p^*=-119$ and  $x^*=({x_1}^*,{x_2}^*,{x_3}^*)=(3,0,8)$.
Characteristics of the three relaxtions methods are collected in the table in Figure~\ref{fig:tab3}.
\begin{figure}[htb]
\begin{center}

\begin{tabular}{|c|c|c|c|}
\hline
              & RLT & LMI  & Blossom \\ \hline
Constraints   & 56+2   & 10$\times$10+8(4$\times$4)&18+2\\
Variables     & 19 &34 &3\\ 
Optimal value &-120 &-119 & -120\\ 
CPU time (sec)    &0.049 & 0.548 &0.024 \\ \hline
\end{tabular} 
\end{center} 
\caption{\label{fig:tab3} Characteristics of the three methods for problem (\ref{eq:test2}).}
\end{figure}

We can see that our approach is the one leading to the simplest relaxation (problem with fewer variables and constraints).
The time needed to compute a lower bound of the optimal value is therefore less than for the other approaches.
The computed lower bound is the same than that computed by the RLT method but not as good as that computed by the LMI method
which finds the exact minimum. Though, on that example, the gap between the computed lower bounds by the three methods is small.
This is unfortunately not always the case as shown on the next example.

Let us consider the following unconstrained 1-dimensional problem~\cite{floudas92} (see also~\cite{sherali97}):
\begin{equation}
\label{eq:test1}
\begin{array}{lll}
\text{min} & y^4-3y^3-1.5y^2+10y\\
\text{over} & y\in [-5,5],
 \end{array}
\end{equation}
Optimal value and solution are $p^*=-7.5$ and  $y^*=-1$.
Characteristics of the three relaxation methods are collected in the table in Figure~\ref{fig:tab1}.
\begin{figure}[htb]
\begin{center}

\begin{tabular}{|c|c|c|c|}
\hline
              & RLT & LMI  & Blossom \\ \hline
Constraints   & 5   & 3$\times$3+2(2$\times$2)& 5\\
Variables     & 4& 4&1 \\ 
Optimal value &-908.3&-7.5&-837.5\\ 
CPU time (sec)    &0.014&0.246 &0.015\\ \hline
\end{tabular} 
\end{center} 
\caption{\label{fig:tab1} Characterisitcs of the three methods for problem (\ref{eq:test1}).}
\end{figure}


The lower bounds provided by the two methods based on linear programming are far from the optimal value.
Notice that for our approach this huge gap is typical from problems where the optimal solution is far from
the boundary of the optimization domain. In that case, coupling our approach with a domain decomposition method based on branch and bound 
as suggested in~\cite{sherali91} would strongly improve our optimal value.
In all cases, the LMI method gives better lower bounds but needs more computational resources.
As for the problem of computing invariants for polynomial systems, it appears, as shown in Section~\ref{sec:example}, that
our linear relaxations are often sufficient.

\section{Computation of Polytopic Invariants for Polynomial Systems}
\label{sec:inv}

In the following, we show how the results developed in the previous section can be used for the computation of polytopic invariants for polynomial  systems.
Let us consider the dynamical system (\ref{eq:ode}),
and let  $R=[\underline{x_1},\overline{x_1}] \times \dots \times [\underline{x_n},\overline{x_n}]$, with $\underline{x_k}<\overline{x_k}$ for all $k\in \{1,\dots, n\}$ be a rectangle of $\R^n$, delimiting a region of interest for studying the dynamics. Our goal is to compute a polytope $P\subseteq R$ invariant for~(\ref{eq:ode}).
To restrict the search space, we shall use parametrized
template expressions for $P$.
We will impose the orientation of the facets of polytope $P$ by choosing normal vectors 
in the set $\{a_k \in \R^n |\; k\in K\}$ where $K=\{1,\dots,m_K\}$ is a set of indices.
Then, polytope $P$ can be written under the form 
$$
P=\{x\in \R^n|\; a_k\cdot x \le b_k,\; \forall k\in K\}
$$
where the vector $b \in \R^{m_K}$, to be determined, specifies the position of the facets. The facets of $P$ are denoted by $F_k$ for $k\in K$, where
$$
F_k=\{x\in \R^n|\; a_k\cdot x =b_k, \text{ and } a_i\cdot x \le b_i,\; \forall i\in K\setminus \{k\}\}.
$$
The proposed approach for the computation of an invariant is iterative. Each iteration consists of two main steps.
First, we try to verify that the polytope $P$ is invariant for the dynamical system~(\ref{eq:ode}). If we fail to verify that $P$ is an invariant, we use sensitivity analysis to modify the vector $b$ (and thus $P$) in order to find an invariant polytope.

\subsection{Polytopic invariant verification}

As stated in the introduction, $P$ is an invariant set of the dynamical system~(\ref{eq:ode}) if and only if 
$$
\forall k\in K,\; \min_{x\in F_k} -a_k\cdot f(x) \ge 0.
$$
We assume that all the facets $F_k$ are not empty. 
Since for all $k\in K$, $F_k \subseteq P \subseteq R$, then this problem is equivalent to showing that the optimal values $p_k^*$ of the following optimization problems
are non-negative for all $k\in K$:
\begin{equation}
\label{eq:inv1}
\begin{array}{llr}
\text{min} & -a_k\cdot f(x)\\
\text{over} & x\in R, \\
\text{under} 
& a_i \cdot x \le b_i, & i\in K\setminus \{k\}, \\
& a_k \cdot x = b_k.
\end{array}
\end{equation}
Since $-a_k\cdot f$ is a multivariate polynomial, this problem is similar to (\ref{eq:opt}). Therefore, by application of Theorem~\ref{th:dual21}, we have the following result:
\begin{proposition}
\label{pro:inv}
For $k\in K$, let  $q_k$ be the blossom of the multivariate polynomial $-a_k\cdot f$ and
let $d_k^*$ be the optimal value of the linear program:
\begin{equation}
\label{eq:inv2}
\begin{array}{llr}
\text{max} & t\\
\text{over} & t\in \R,\; \lambda\in \R^{m_K},\\
\text{under} &\lambda_i \ge 0, & i\in K\setminus \{k\}, \\ 
&t \le q_k(\overline{v})+\displaystyle{\sum_{i\in K} \lambda_i (a_i' \cdot \overline{v} - b_i)}  & \overline{v}\in \overline{V}'.
\end{array}
\end{equation}
If for all $k\in K$, $d_k^*\ge 0$, then $P$ is an invariant polytope for dynamical system~(\ref{eq:ode}).
\end{proposition}

\begin{proof} By applying Theorem~\ref{th:dual21} to problem (\ref{eq:inv1}), we obtain that $p_k^* \ge d_k^*$, for all $k\in K$. Then,  $d_k^*\ge 0$ implies that 
$p_k^*\ge 0$, for all $k\in K$, and therefore $P$ is an invariant polytope.
\end{proof}

\begin{remark}
The degrees of the multivariate polynomials $-a_k\cdot f$ may not be all the same. This results in vectors $a_i'$, and sets $\overline{V}'$ in problem~(\ref{eq:inv2}) that depend on the index $k\in K$. It is possible to avoid this by defining $q_k$ as $-a_k \cdot g$ where $g$ is the blossom of the polynomial vector field $f$ defined as follows. For  $ i,j \in \{1,\ldots,n\} $,
let $\delta_{ij}$ be the degree of $x_i$ in the multivariate polynomial $f_j(x)$ and let $\delta_i={\max_{j\in \{1,\ldots,n \} } \delta_{ij}}$.
It is possible to regard $f_j$ as a multivariate polynomial with degrees $\delta_1,\dots,\delta_n$ possibly with some zero coefficients and to define the associated blossom $g_j$ as defined in Definition~\ref{def:blossom}. Then, for $j\in \{1,\ldots,n\} $, $g_j$ are the components of the blossom 
$g: \R^{\delta_1+\dots+\delta_n} \rightarrow \R^n$ of the polynomial vector field $f$. For $k\in K$, $q_k=-a_k\cdot g$ are multi-affine functions defined on $\R^{\delta_1+\dots+\delta_n}$ with similar properties to the blossom of $-a_k\cdot f$ that can be used in problem~(\ref{eq:inv2}).
\end{remark}

Invariance of a polytope $P$ can be verified by solving a set of linear programs (one per facet of $P$). In the case when we fail to verify that the polytope is invariant, sensitivity analysis may help us in modifying it in order to find an invariant polytope for~(\ref{eq:ode}).

\subsection{Polytope modification using sensitivity analysis} The verification of the invariance of the polytope $P$ fails
if $d_k^*< 0$, for some $k\in K$. 
In that case, we would like to know how
to modify the vector $b$ (and thus $P$) in order to find an invariant polytope. 
For 
$\alpha\in R^{m_K}$, let $P_\alpha$ be the polytope given by
$$
P_\alpha=\{x\in \R^n|\; a_k\cdot x \le b_k+\alpha_k,\; \forall k\in K\}.
$$
For $\alpha=0$, we recover the polytope $P$, we would like to find $\alpha$ such that $P_\alpha$ is an invariant for~(\ref{eq:ode}). 
We impose additional constraints on $P_\alpha$:
\begin{itemize}
\item constraints of  form $b_k+\alpha_k \le \overline{b_k}$ ensures  $P_\alpha \subseteq R$;

\item constraints of  form $\underline{b_k} \le b_k+\alpha_k$ ensures $P_\alpha\ne \emptyset$;

\item $-\varepsilon \le \alpha_k \le \varepsilon$ ensures that $P_\alpha$ is close to $P$, where $\varepsilon$ is a parameter that can be tuned.
\end{itemize}

Denoting for $k\in K$, $d_k^*(\alpha)$ the optimal values of problems (\ref{eq:inv2}) for the polytope $P_\alpha$, the sensitivity analysis 
 in Theorem~\ref{th:sens} gives us 
$
d_k^*(\alpha) \ge d_k^* +\lambda_k^* \cdot \alpha
$,
for all $k\in K$. 
where $d_k^*$ and $(t_k^*,\lambda_k^*)$ are the optimal values and solutions of problems (\ref{eq:inv2}) for polytope $P$ and $k\in K$. 
Then, by Proposition~\ref{pro:inv}, for $P_\alpha$ to be an invariant polytope for dynamical system (\ref{eq:ode}), it is sufficient that for all $k\in K$, $d_k^* +\lambda_k^* \cdot \alpha \ge 0$. 
In order to find a suitable $\alpha$, we can solve the following problem:
$$
\begin{array}{llr}
\text{max} & \displaystyle{\min_{k\in K}    \left(d_k^* +\lambda_k^* \cdot \alpha\right)   }\\
\text{over} & \alpha \in \R^{m_K},\\
\text{under}
& \underline{\alpha_k} \leq \alpha_k \leq \overline{\alpha_k}, & k\in K
\end{array}
$$
where $\underline{\alpha_k}=\max(-\varepsilon, \underline{b_k}-b_k)$ and $\overline{\alpha_k}=\min(\varepsilon, \overline{b_k}-b_k)$.
This problem can be recasted as the following linear program:
\begin{equation}
\label{eq:inv3}
\begin{array}{llr}
\text{max} & t\\
\text{over} & t\in \R,\;  \alpha \in \R^{m_K},\\
\text{under}  
&t \le d_k^*-{\lambda_k}^*\cdot \alpha, &  k\in K,\\
& \underline{\alpha_k} \leq \alpha_k \leq \overline{\alpha_k}, & k\in K.
\end{array}
\end{equation}
Let $(t^*,\alpha^*)$ be an optimal solution of this linear program. If the optimal value is non-negative then it is sufficient to prove that $P_{\alpha^*}$ is an invariant for the dynamical system~(\ref{eq:ode}).
If the optimal value is strictly negative, then we go back to the verification stage with $P=P_{\alpha^*}$ and iterate the approach.

\begin{remark} Let us remark that the polytope $P_{\alpha^*}$ computed by solving (\ref{eq:inv3}) may have empty facets. This results, for the empty facet $F_k$, in an unbounded value $d_k^*(\alpha^*)=+\infty$. In order to avoid such situations, it is useful to replace $\alpha^*$ by $\tilde{\alpha}^*$ such that $P_{\tilde{\alpha}^*}$ has no empty facet and $P_{\alpha^*}= P_{\tilde{\alpha}^*}$ (see Figure~\ref{fig:constraints}). Again, this can be done by solving a set of linear programs. 
\end{remark}

\begin{figure}[!h]
\begin{center}
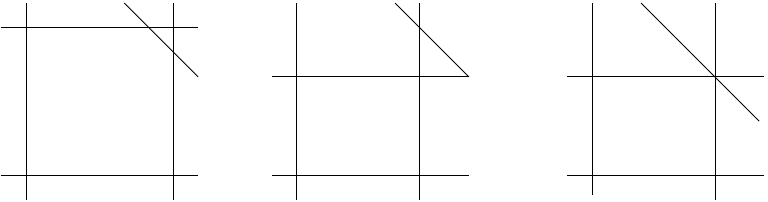
\caption{The polytope $P_{\alpha^*}$ may have empty facets (center polytope), we replace $\alpha^*$ by $\tilde{\alpha}^*$ such that $P_{\tilde{\alpha}^*}$ has no empty facet and $P_{\alpha^*}= P_{\tilde{\alpha}^*}$ (right polytope). \label{fig:constraints}}
\end{center}
\end{figure}

\subsection{Related work}

Computation of invariants for polynomial dynamical systems is often approached using semi-definite programming via sum of squares relaxations.
Most of the literature on the subject deals with the computation of Lyapunov functions, whose level sets are invariant (see e.g.~\cite{jarvis2003,wang2005}). A similar approach for the computation of invariant sets that do not contain any stable equilibrium point can  be found in~\cite{Prajna07}.
In these works, the invariant sets are semi-algebraic sets described by polynomial inequalities. However, polytopes are easier to manipulate and as explained in~\cite{Alessio2007}, it is sometimes preferable to have invariants described by polytopes rather than semi-algebraic sets. In particular, polytopes with fixed facet directions given by a template has been shown very useful for the computation of invariant sets for hybrid systems with affine dynamics in~\cite{Sankaranarayanan2008} and for multiaffine systems in~\cite{Belta06,Abate09}. In some sense, our work builds on and extend these approaches to the class of polynomial dynamical systems.
  
It is worth mentioning the work on polyhedral Lyapunov functions for the class of dynamical systems described by linear differential inclusions. 
For this class of systems, it can be shown that existence of an invariant set containing the origin is equivalent to Lyapunov stability~\cite{blanchini99}.
Methods relying on linear programming, for computing polyhedral invariants for linear differential inclusions have been developed (see 
e.g.~\cite{blanchini95,polanski2000}). We would like to point out that our approach can be easily adapted for the computation of invariant sets for polynomial (and thus linear) differential inclusions. 
Hence, our approach can be seen as a generalization of the work mentioned above though the algorithms used for the computation of the polyhedral
invariants are quite different.


\section{Examples}\label{sec:example} We implemented our approach in Matlab; in the following, we show for a set of examples borrowed from biological applications, that our approach is effective in practice. All the reported computations take a few seconds.

\subsection{FitzHugh-Nagumo neuron model} 

We applied our approach to the FitzHugh-Nagumo model~\cite{Fitzhugh1961}, a polynomial dynamical system modelling the electrical activity of a neuron:
\begin{equation*}
\left\{
\begin{array}{rcll}
 \dot{x}_1 &=& x_1-x_1^3/3 - x_2 + I , \vspace{2mm}\\
 \dot{x}_2 &=& 0.08(x_1+0.7 - 0.8x_2),
\end{array}\right.
\end{equation*}
where model parameter $I$ is taken equal to $\frac{7}{8}$. This system is known to have a limit cycle. Using our approach, we synthesized an invariant polytope containing the limit cycle. Working in the rectangle $[-2.5,2.5]\times [-1.5, 3.5]$, we found an invariant polytope with $8$  facets with uniformly distributed orientations (see Figure~\ref{fig:ex1}). Starting from the set represented in dashed line, our approach needs $15$ iterations to find the invariant polytope depicted in plain line. We can check on the figure that it is effectively an invariant. Let us remark that this invariant polytope $P$ together with the existence of an unstable equilibrium inside $P$ provides
by application of the Poincar\'e Bendixon theorem a formal proof of the existence of a limit cycle inside the polytope $P$. 

At each iteration, we need to solve for invariant verification, $m_K=8$ linear programs of the form (\ref{eq:inv2}) with $m_K+1=9$ variables and $m_K-1+(\delta_1+1)\times(\delta_2+1)=15$ inequality constraints. For invariant synthesis using sensitivity analysis, we need to solve at each iteration $1$ linear program with $m_K+1=9$ variables and $2m_K=16$ inequality constraints.


\begin{figure}[!h]
\begin{center}
\includegraphics[angle=0,scale=0.4]{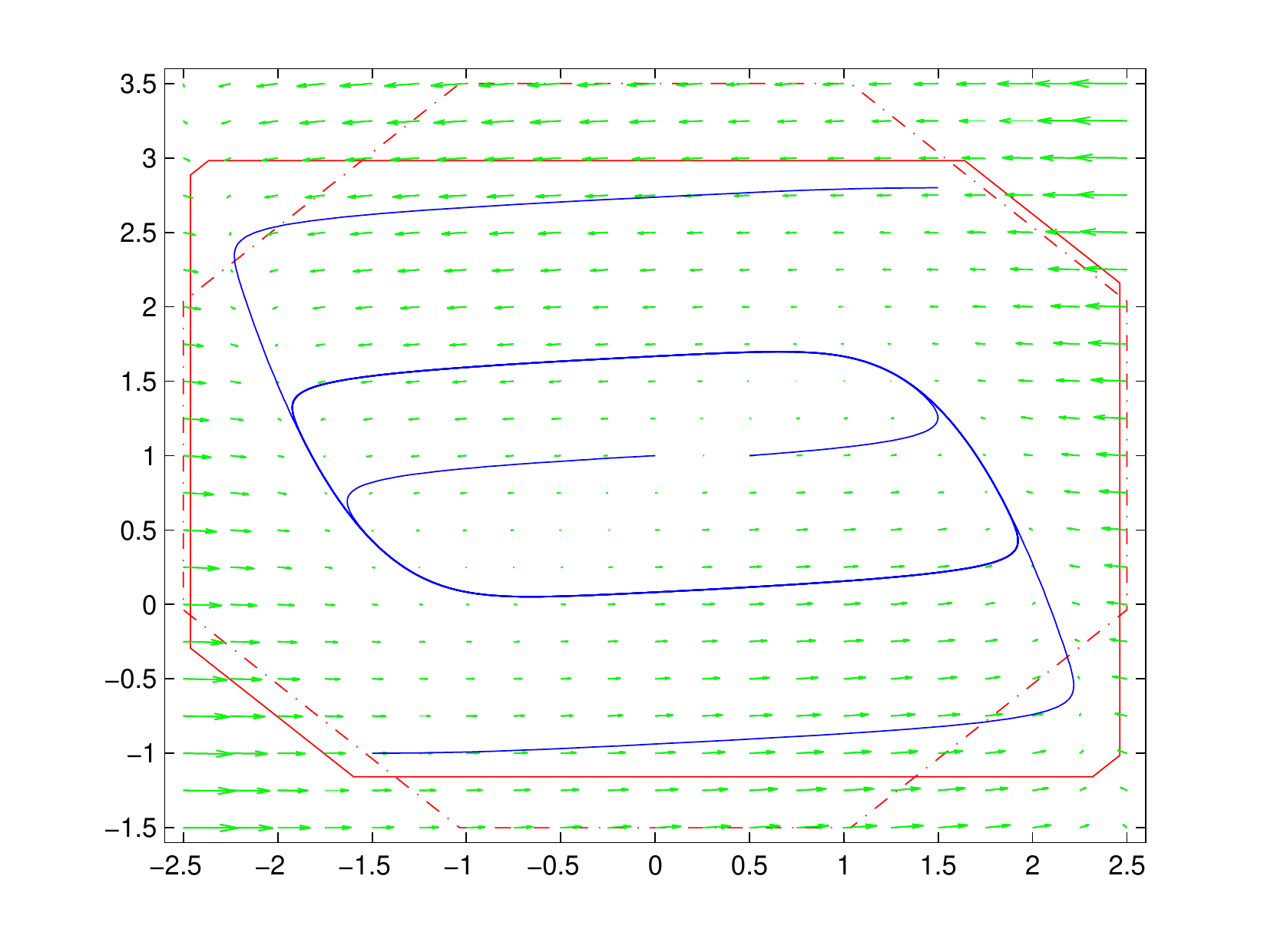}
\caption{Polytopic invariant for the FitzHugh-Nagumo model (represented in plain line) obtained after $15$ iterations starting from the dashed polytope. The computed invariant contains the limit cycle.
\label{fig:ex1}}
\end{center}
\end{figure}

\subsection{Phytoplankton growth model}
We now consider a model of Phytoplankton growth~\cite{Bernard2002}:
\begin{equation*}
\left\{
\begin{array}{rcll}
 \dot{x}_1 &=& 1-x_1-\frac{x_1x_2}{4}, \vspace{2mm}\\
 \dot{x}_2 &=& (2x_3-1)x_2,\vspace{2mm}\\
 \dot{x}_3 &=& \frac{x_1}{4}-2x_3^2.
\end{array}\right.
\end{equation*}
This system has a stable equilibrium. Using our approach, we synthesized an invariant polytope containing the equilibrium. Working  in the rectangle 
$[0,3]\times [-0.1,2] \times [0,0.6]$, we were able to find an invariant polytope with $m_K=18$ facets a regular octagon (see Figure~\ref{fig:ex3d}). 
Starting from the polytope represented in left part of the figure, our approach needs $11$ iterations to find the invariant polytope depicted in the right part of the figure. We can check on the figure that it is indeed an invariant. 

At each iteration, we need to solve for invariant verification, $m_K=18$ linear programs of the form (\ref{eq:inv2}) with $m_K+1=19$ variables and $m_K-1+(\delta_1+1)\times(\delta_2+1)\times(\delta_3+1)=29$ inequality constraints. For invariant synthesis using sensitivity analysis, we need to solve at each iteration $1$ linear program with $m_K+1=19$ variables and $2m_K=36$ inequality constraints.

\begin{figure}[!h]
\begin{center}
\includegraphics[angle=0,scale=0.4]{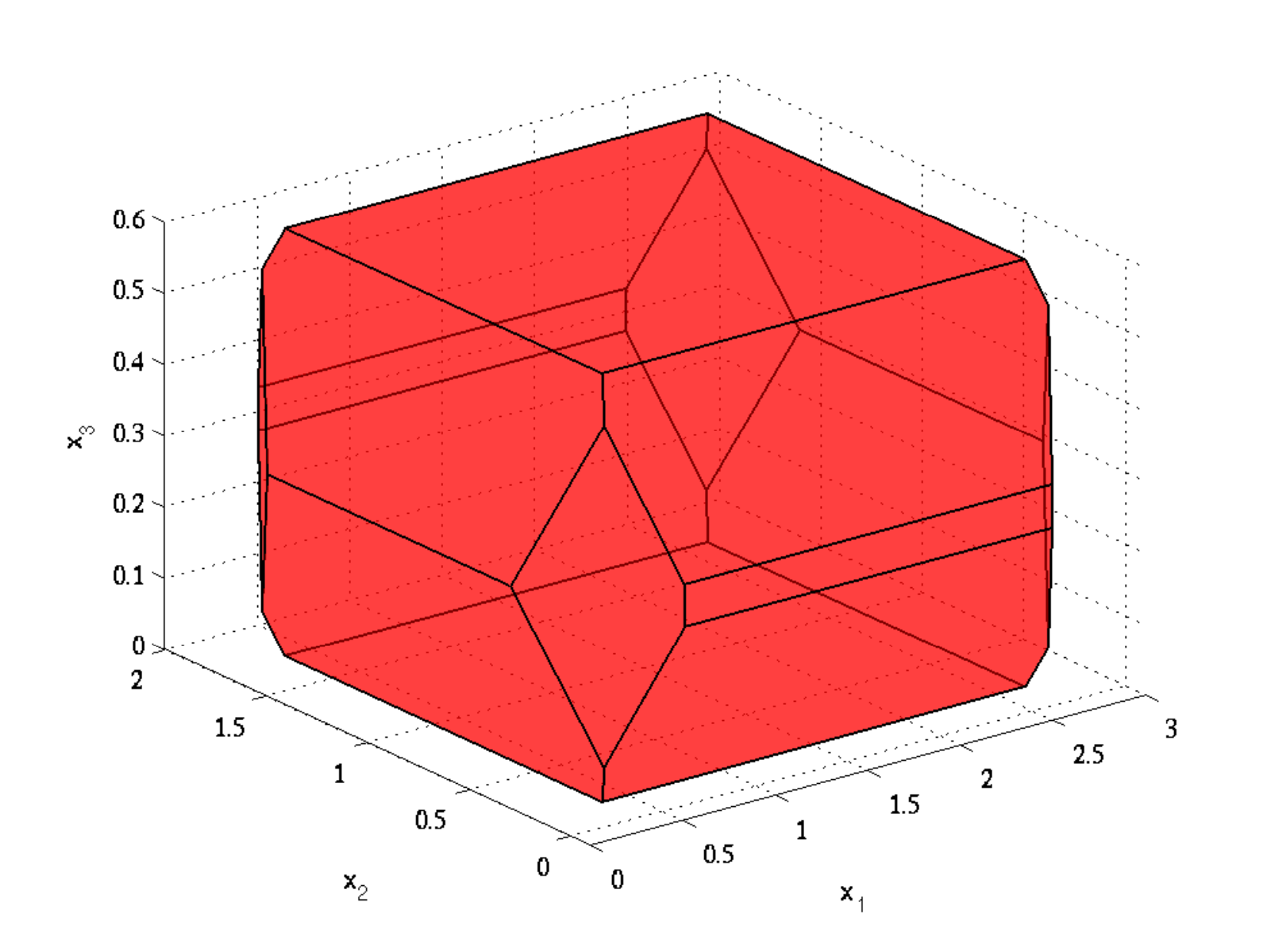}\\ \includegraphics[angle=0,scale=0.4]{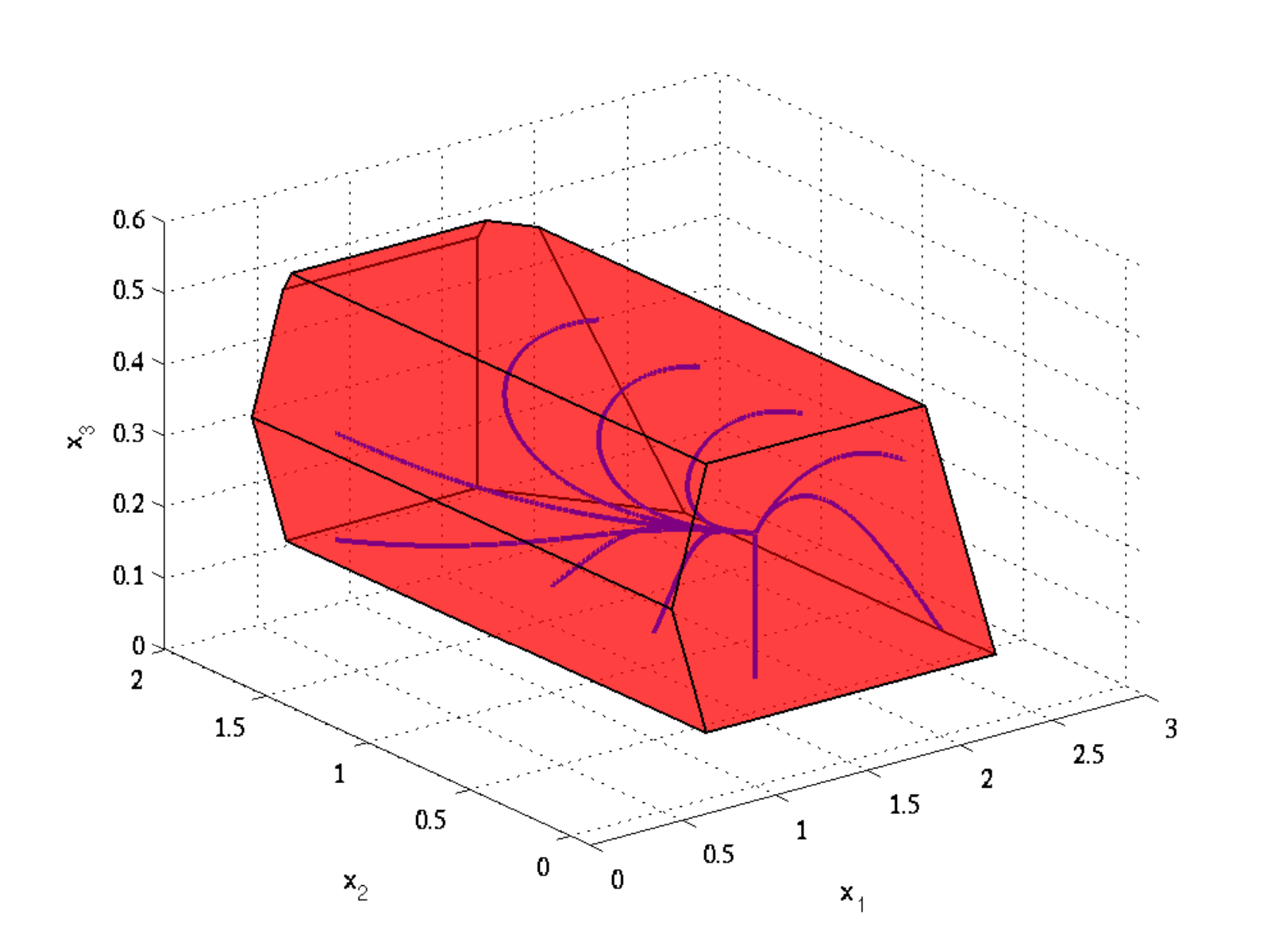}
\caption{Polytopic invariant for the Phytoplankton growth model (on the right) obtained after $11$ iterations starting from the polytope on the left.
\label{fig:ex3d}}
\end{center}
\end{figure}


\section{Conclusion}

In this paper, we have presented an approach based on linear programming for the computation of polytopic invariants for polynomial dynamical systems.
It uses the blossoming principle for polynomials, properties of multi-affine functions, Lagrangian duality and sensitivity analysis. Though our approach is conservative (we may fail to verify invariance of a polytope), we have shown on several examples that it can be useful in practical applications. 
In the future, we plan to use similar ideas for control synthesis for polynomial dynamical systems.


\bibliographystyle{alpha}
\bibliography{ref}

\section*{Appendix - Proof of Theorem~\ref{th:dual21}}

Let ${d}^*_1$ be the optimal value of (\ref{eq:dual11}) and $d_2^*$ the optimal value
of (\ref{eq:dual21}), we want to show that $d_1^*=d_2^*$. 
We start by remarking that (\ref{eq:dual11}) and (\ref{eq:dual21}) are linear programs therefore
their optimal values are equal to that of their dual problems.
One can show verify the dual
of (\ref{eq:dual11}) is
\begin{equation}
\label{eq:ddual1}
\begin{array}{lll}
\text{min} & \displaystyle{\sum_{v\in V'} y_v q(v)} \\
\text{over} & y \in \R^{2^{\delta_1+\dots+\delta_n}},  \\
\text{under} & y_v \ge 0, & v\in V', \\
&  \displaystyle{\sum_{v\in V'} y_v = 1},\\
&  \displaystyle{a_i' \cdot  \sum_{v\in V'} y_v v \le b_i},& i\in I, \\
&  \displaystyle{c_j' \cdot \sum_{v\in V'} y_v v   = d_j},& j\in J, \\
&  \displaystyle{ e_{k,l} \cdot \sum_{v\in V'} y_v v   = 0},& k\in\{1,\dots,n\},\\
& & l\in \{1,\dots,\delta_k-1\}. 
\end{array}
\end{equation}
Similarly, the dual of (\ref{eq:dual21}) is
\begin{equation}
\label{eq:ddual2}
\begin{array}{llr}
\text{min} & \displaystyle{\sum_{\overline{v}\in \overline{V}'} z_{\overline{v}} q(\overline{v})} \\
\text{over} & z \in \R^{(\delta_1+1)\times \dots \times (\delta_n+1)},  \\
\text{under} & z_{\overline{v}} \ge 0, & \overline{v}\in \overline{V}', \\
&  \displaystyle{\sum_{\overline{v}\in \overline{V}'} z_{\overline{v}} = 1},\\
&  \displaystyle{a_i' \cdot  \sum_{\overline{v}\in \overline{V}'} z_{\overline{v}} \overline{v} \le b_i},& i\in I, \\
&  \displaystyle{c_j' \cdot \sum_{\overline{v}\in \overline{V}'}  z_{\overline{v}} \overline{v}   = d_j},& j\in J. 
\end{array}
\end{equation}
We first show that $d_2^* \le d_1^*$. Let $y\in \R^{2^{\delta_1+\dots+\delta_n}}$ be a feasible point
for problem (\ref{eq:ddual1}) such that ${d}_1^*={\sum_{v\in V'} {y_v}q(v)}$.
For $\overline{v} \in \overline{V}'$, let $z_{\overline{v}}={\sum_{v\in \overline{v}}y_v}$, it is clear that $z_{\overline{v}} \ge 0$. Further,
$$
\sum_{\overline{v}\in \overline{V}'} z_{\overline{v}} = \sum_{\overline{v}\in \overline{V}'}
\sum_{v\in \overline{v}}y_v = \sum_{v\in V'}y_v=1,
$$
$$
a_i' \cdot  \sum_{\overline{v}\in \overline{V}'} z_{\overline{v}} \overline{v} =
 \sum_{\overline{v}\in \overline{V}'} \sum_{v\in \overline{v}}y_v (a_i'\cdot \overline{v}) = 
 \sum_{{v}\in {V'}} y_{{v}} (a_i'\cdot {v}) \le b_i,
$$
and
$$
c_j' \cdot  \sum_{\overline{v}\in \overline{V}'} z_{\overline{v}} \overline{v} =
 \sum_{\overline{v}\in \overline{V}'} \sum_{v\in \overline{v}}y_v (c_j'\cdot \overline{v}) = 
 \sum_{{v}\in {V'}} y_{{v}} (c_j'\cdot {v}) = d_j.
$$
Therefore, $z$ is feasible for problem (\ref{eq:ddual2}).
Finally, since for all $v \in \overline{v}$, $q(v)=q(\overline{v})$, it follows that
$$
\sum_{\overline{v}\in \overline{V}'} z_{\overline{v}} q(\overline{v}) =
 \sum_{\overline{v}\in \overline{V}'} \sum_{v\in \overline{v}}y_v q(\overline{v}) = 
 \sum_{{v}\in {V'}} y_{{v}} q(v) = d_1^*.
$$
Therefore, $d_2^* \le d_1^*$. 
We now show that $d_1^* \le d_2^*$.
Let $z \in \R^{(\delta_1+1)\times \dots \times (\delta_n+1)}$ be a feasible point
for problem (\ref{eq:ddual2}) such that $d_2^*={\sum_{\overline{v}\in \overline{V}'} z_{\overline{v}} q(\overline{v})}$. Let $n(\overline{v})$ denote the number of vertices $v \in \overline v$, 
then for all $v \in \overline{v}$, let $y_v=z_{\overline{v}}/n(\overline{v})$.
It is clear $y_v\ge 0$ and
$$
\sum_{v\in V'}y_v = 
\sum_{\overline{v}\in \overline{V}'} \sum_{v\in \overline{v}} y_v =
 \sum_{\overline{v}\in \overline{V}'} z_{\overline{v}}
=1.
$$
We also have that
\begin{eqnarray*}
a_i' \cdot  \sum_{v\in V'} {y_v} v&= &\sum_{\overline{v}\in \overline{V}'} \sum_{v\in \overline{v}}
{y_v} (a_i' \cdot v) = \sum_{\overline{v}\in \overline{V}'} \sum_{v\in \overline{v}}
\frac{z_{\overline{v}}}{n(\overline{v})} (a_i' \cdot \overline{v})\\& =&
\sum_{\overline{v}\in \overline{V}'} z_{\overline{v}} (a_i' \cdot \overline{v}) \le b_i,
\end{eqnarray*}
and similarly we can show
$
c_j' \cdot  \sum_{v\in V'} {y_v} v = d_j.
$
%
Further, 
\begin{eqnarray*}
e_{k,l} \cdot \sum_{v\in V'} {y_v} v &= &\sum_{\overline{v}\in \overline{V}'} e_{k,l} \cdot \sum_{v\in \overline{v}}
\frac{z_{\overline{v}}}{n(\overline{v})} v \\&=&
\sum_{\overline{v}\in \overline{V}'} \frac{z_{\overline{v}}}{n(\overline{v})}\left( e_{k,l} \cdot \sum_{v\in \overline{v}} 
 v \right).
\end{eqnarray*}
By remarking, that for all $\overline{v} \in \overline{V}'$,  
$e_{k,l} \cdot \sum_{v\in \overline{v}} v = 0$, it follows that 
$e_{k,l} \cdot \sum_{v\in V'} {y_v} v = 0$.
Therefore, $y$ is feasible for problem (\ref{eq:ddual1}).
Finally,
$$
\sum_{{v}\in {V'}} y_{{v}} q(v) = 
 \sum_{\overline{v}\in \overline{V}'} \sum_{v\in \overline{v}} \frac{z_{\overline{v}}}{n(\overline{v})} q(\overline{v}) = 
\sum_{\overline{v}\in \overline{V}'} z_{\overline{v}} q(\overline{v}) = d_2^*.
$$
This proves that $d_1^* \le d_2^*$ and $d^*=d_1^*=d_2^*$. The fact that $d^* \le p^*$ is a consequence of Proposition~\ref{pro:dual}.

\end{document}